\documentclass[conference]{IEEEtran}
\IEEEoverridecommandlockouts
\usepackage{cite}
\usepackage{amsmath,amssymb,amsfonts}
\usepackage{algorithmic}
\usepackage{graphicx}
\usepackage{textcomp}
\usepackage{xcolor}
\usepackage{amsthm} 

\newtheoremstyle{italicheadplainbody} 
  {\topsep}   
  {\topsep}   
  {\normalfont} 
  {}          
  {\itshape}  
  {.}         
  {.5em}      
  {}          

\theoremstyle{italicheadplainbody}

\newtheorem{assumption}{Assumption}
\newtheorem{lemma}{Lemma}

\def\BibTeX{{\rm B\kern-.05em{\sc i\kern-.025em b}\kern-.08em
    T\kern-.1667em\lower.7ex\hbox{E}\kern-.125emX}}
\begin{document}

\title{Approximate Solution Methods for the Average Reward Criterion in Optimal Tracking Control of Linear Systems\\
}

\author{\IEEEauthorblockN{ Duc Cuong Nguyen}
\IEEEauthorblockA{\textit{Automatic Control LTH,} 
\textit{Lund University}\\
Lund, Sweeden \\
}
}

\maketitle

\begin{abstract}
This paper studies optimal control under the average-reward/cost criterion for deterministic linear systems. We derive the value function and optimal policy, and propose an approximate solution using Model Predictive Control to enable practical implementation.
\end{abstract}

\section{Introduction}
The average reward formulation has recently attracted increasing attention in reinforcement learning, particularly in the context of continuing tasks \cite{sutton}. Unlike the commonly used discounted reward approach \cite{zhang2021policy}, the average reward criterion offers a hyperparameter-free alternative that naturally captures long-term system behavior \cite{sutton}. This advantage becomes especially significant in settings such as inverse reinforcement learning \cite{wu2023inverse}, where the need to specify a discount factor to model expert behavior can introduce ambiguity. Motivated by these benefits, this paper explores the application of the average reward framework to deterministic tracking control problems \cite{miquel:hal-02987731}, where traditional cost functions often diverge and lack finite solutions. We propose a novel formulation of the tracking control problem—an area with extensive prior literature and practical relevance—under the average cost setting. Our contributions include an introductory analysis of the value function solution to the Bellman equation for general linear systems. Furthermore, we develop a practical, approximate solution method based on Model Predictive Control (MPC), enabling real-time implementation in control applications.

\section{Main result}
Consider the following linear system:
\begin{equation}\label{linear}
    \begin{cases}
          x_{k+1} = Ax_k +Bu_k\\
          y_{k}= Cx_k
    \end{cases}
\end{equation}
where $A \in \mathbb{R}^{n \times n}$, $B \in \mathbb{R}^{n \times m}$, $C \in \mathbb{R}^{p \times n}$.
This study focuses on tracking a constant reference signal $r_k= r_{ss} \in \mathbb{R}^{p}$  over an infinite horizon.
The natural stage cost for this problem is defined as \( C(e_k, u_k) = e_k^\top Q e_k + u_k^\top R u_k \), where \( e_k = Cx_k - r_k \), and \( Q, R \) are positive definite matrices. However, the infinite-horizon sum of this cost becomes divergent, even as \( e_k \to 0 \) for \( k \to \infty \), due to the control input \( u_k \) converging to a nonzero steady-state value. Rather than applying a discount factor as in \cite{modares2014linear}, we propose the following modified cost index:
\begin{equation}
\begin{aligned}
    J &= \sum_{k=0}^\infty \left| x_k^\top C^\top Q C x_k + u_k^\top R u_k - x_{ss}^\top C^\top Q Cx_{ss} - u_{ss}^\top R u_{ss} \right| \\
    &= \sum_{k=0}^\infty \left| C(x_k, u_k) - C_{ss} \right|,
    \end{aligned}
\end{equation}
where \( C_{ss} = C(Cx_{k\rightarrow\infty}, u_{k\rightarrow\infty}) \) denotes the steady-state cost which is analogous to the Average-Reward definition $\underset{N \rightarrow \infty}{\lim}\frac{1}{N}\mathbb{E}[\sum_{k=0}^N C(x_k,u_k)]$ in a stochastic setting.
The cost function proposed can be rewritten as following:
\begin{equation}
\begin{aligned}\label{true_avg}
     J &= \sum_{k=0}^\infty \left| \tilde{x}_k^\top C^\top Q C\tilde{x}_k + \tilde{u}_k^\top R \tilde{u}_k + x_{ss}^\top C^\top Q \tilde{x}_{k} + u_{ss}^\top R \tilde{u}_{k} \right|  \\
     &= \sum_{k=0}^\infty \left| \tilde{x}_k^\top C^\top Q C \tilde{x}_k + \tilde{u}_k^\top R \tilde{u}_k + s^\top\tilde{x}_{k} + r^\top \tilde{u}_{k} \right| 
\end{aligned}
\end{equation}
where $\tilde{x}_k = x-x_{ss}$, $\tilde{u}_k= u_k- u_{ss}$

Note that the proposed cost index shares structural similarities with the optimal tracking cost presented in \cite{miquel:hal-02987731}, except for the absence of the linear term. Under the same assumptions as in \cite{miquel:hal-02987731}, we will show that there exists a value function \( V(\tilde{x}_k) \) that satisfies the Bellman equation.

\begin{assumption}
The pair $(A,\sqrt{Q}C)$ is observable and $(A,B)$ is controllable.
\end{assumption}
\begin{assumption}The matrix 
\begin{equation}
\begin{bmatrix}
    A - I & B \\
    C & \mathbf{0}
\end{bmatrix}
\end{equation}
is assumed to be invertible, where \( I \) is the identity matrix and \( \mathbf{0} \) is the zero matrix, both of appropriate dimensions.
\end{assumption}

\begin{lemma}
Under Assumption 1, there exists a value function \( V(\tilde{x}_k) \) that satisfies the Bellman equation:
\begin{equation}
    V(\tilde{x}_k) = \min_{\{u_k\}_{k=0}^\infty} \left\{ \left| C(x_k, u_k) - C_{ss} \right| + V(\tilde{x}_{k+1}) \right\}.
\end{equation}
\end{lemma}

\begin{proof}
To track the reference signal over an infinite horizon, the following steady-state condition must be satisfied as \( k \to \infty \):
\begin{equation}
    \begin{bmatrix}
        A - I & B \\
        C & \mathbf{0}
    \end{bmatrix}
    \begin{bmatrix}
        x_{ss} \\
        u_{ss}
    \end{bmatrix}
    =
    \begin{bmatrix}
        \mathbf{0} \\
        r_{ss}
    \end{bmatrix}.
\end{equation}

The dynamics of the deviation state become:
\begin{equation}
    \tilde{x}_{k+1} = A \tilde{x}_k + B \tilde{u}_k.
\end{equation}

By Assumption 1, there exists a feedback gain \( K \) such that the closed-loop matrix \( A + BK \) is Schur stable (i.e., all eigenvalues lie inside the unit circle). Suppose the spectral radius satisfies \( \rho < 1 \). Then there exists a constant \( M > 0 \) such that:
\begin{equation}
    \| \tilde{x}_k \| \le M \rho^k \| \tilde{x}_0 \|, \quad
    \| \tilde{u}_k \| = \| K \tilde{x}_k \| \le M \|K\| \rho^k \| \tilde{x}_0 \|.
\end{equation}

Now consider the stage cost deviation:
\begin{equation}
    \phi_k := \tilde{x}_k^\top C^\top Q C \tilde{x}_k + \tilde{u}_k^\top R \tilde{u}_k + s^\top \tilde{x}_k + r^\top \tilde{u}_k,
\end{equation} 
Each term is quadratic or linear in \( \tilde{x}_k \) or \( \tilde{u}_k \), so we can bound it by:
\begin{equation}\label{abs_upper}
    | \phi_k | \le \lambda_1 \| \tilde{x}_k \|^2 + \lambda_2 \| \tilde{u}_k \|^2 + \lambda_3 \| \tilde{x}_k \| + \lambda_4 \| \tilde{u}_k \|,
\end{equation}
for some constants \( \lambda_1, \lambda_2, \lambda_3, \lambda_4 > 0 \).

Using the exponential decay, we obtain:
\begin{equation}
    | \phi_k | \le C_1 \rho^{2k} \| \tilde{x}_0 \|^2 + C_2 \rho^k \| \tilde{x}_0 \|,
\end{equation}
for some constants \( C_1, C_2 > 0 \). Therefore, the infinite-horizon cost is bounded:
\begin{equation}
\begin{aligned}
    J(\tilde{x}_0) &\le \sum_{k=0}^{\infty} \left( C_1 \rho^{2k} \| \tilde{x}_0 \|^2 + C_2 \rho^k \| \tilde{x}_0 \| \right) \\
    &= \| \tilde{x}_0 \|^2 \cdot \frac{C_1}{1 - \rho^2} + \| \tilde{x}_0 \| \cdot \frac{C_2}{1 - \rho} < \infty.
\end{aligned}
\end{equation}

Thus, a stabilizing admissible control policy exists that yields finite cost. The optimal value function is:
\[
V(\tilde{x}_0) = \inf_{\{ \tilde{u}_k \}} J(\tilde{x}_0),
\]
and is finite for all \( \tilde{x}_0 \). By standard dynamic programming theory, it follows that:
\begin{itemize}
    \item \( V(\tilde{x}) \) is finite and continuous,
    \item There exists an optimal stationary policy \( \tilde{u}_k = \pi(\tilde{x}_k) \),
    \item \( V(\tilde{x}) \) satisfies the Bellman equation:
    \begin{equation}
        V(\tilde{x}_k) = \min_{\{u_k\}_{k=0}^\infty} \left\{ \left| C(x_k, u_k) - C_{ss} \right| + V(\tilde{x}_{k+1}) \right\}.
    \end{equation}
\end{itemize}
\end{proof}

\noindent
By Assumption 2, the optimal controller for the deviation system can be transformed into a controller for the original dynamics. Specifically, the steady-state pair \( (x_{ss}, u_{ss}) \) is obtained by solving:
\begin{equation} \label{steady_state}
    \begin{bmatrix}
        x_{ss} \\
        u_{ss}
    \end{bmatrix}
    =
    \begin{bmatrix}
        A - I & B \\
        C & \mathbf{0}
    \end{bmatrix}^{-1}
    \begin{bmatrix}
        \mathbf{0} \\
        r_{ss}
    \end{bmatrix}.
\end{equation}

Then, the optimal tracking control law becomes:
\begin{equation}
    u_k = \tilde{u}_k + u_{ss}.
\end{equation}

Although Lemma 1 establishes the existence of a value function, the resulting function is piecewise and lacks a general closed-form expression for arbitrary linear systems. Motivated by the construction used in the proof of Lemma 1, we propose an approximate solution by minimizing an upper bound of the stage cost, formulated as follows:

\begin{equation}\label{upper}
    \tilde{J} = \sum_{k=0}^\infty  \tilde{x}_k^\top C^\top Q C \tilde{x}_k + \tilde{u}_k^\top R \tilde{u}_k +\left| s^\top\tilde{x}_{k} + r^\top \tilde{u}_{k} \right|
\end{equation}

Note that the upper bound of the cost in \eqref{upper} remains finite for the same reasons established in the proof above. By minimizing this upper bound, we also implicitly minimize the original average-cost index. While this approximation does not guarantee an optimal solution—and may yield a suboptimal one—it offers a practical advantage: the separation of quadratic and linear terms allows the cost to be handled using linear programming, which is well-suited for MPC implementation. Therefore, we view this formulation as a trade-off between computational accuracy and implementation feasibility. It is worth noting that various forms of upper bounds, such as the one in \eqref{abs_upper}, can be used to approximate the average-cost index. However, care must be taken not to overestimate the cost too aggressively, as the MPC-generated solution is itself only an approximation of the true optimal solution.
The MPC formulation for this problem is given by:
\begin{equation}\label{MPC_cost}
\begin{aligned}
& \underset{u_k, k={t,\cdots},t+N-1}{\text{minimize}}
&\tilde{J}_k= \sum_{k=t}^{t+N-1} &  \bigr[\tilde{x}_k^\top C^\top Q C \tilde{x}_k + \tilde{u}_k^\top R \tilde{u}_k \\
&&&+\left|  s^\top\tilde{x}_{k} + r^\top \tilde{u}_{k} \right|\bigr] + \tilde{J}_N
\\
& \text{subject to}
&\tilde{x}_{k+1}&=A\tilde{x}_k+B\tilde{u}_k
\end{aligned}
\end{equation}

The terminal cost \( \tilde{J}_N \) in \eqref{MPC_cost} can be approximated in various ways, as discussed in \cite{johansson2024stablempcmaximalterminal, kohler2021stability, moreno2023predictive}. In this work, we adopt a simple and practical approach by rolling out the traditional Linear Quadratic Regulator (LQR) optimal feedback policy \( \hat{u}_k = -K x_k \) over a finite prediction horizon \( h \). The resulting terminal cost is given by:
\begin{equation}
\tilde{J}_N = \sum_{k=t+N}^{h} \tilde{x}_k^\top C^\top Q C \tilde{x}_k + \hat{u}_k^\top R \hat{u}_k + \left| s^\top \tilde{x}_k + r^\top \hat{u}_k \right|
\end{equation}

\section{Scalar Example}

Consider the scalar system:
\begin{equation}
\begin{cases}
    x_{k+1} = 2x_k + u_k, \\
    y_k = x_k,
\end{cases}
\end{equation}
with the objective of tracking the constant reference signal \( r_k = 1 \), and stage cost parameters \( Q = 1 \), \( R = 1 \).

The cost function is defined as:
\begin{equation}\label{ex_upper}
    \tilde{J} = \sum_{k=0}^\infty \tilde{x}_k^2 + \tilde{u}_k^2 + \left| \tilde{x}_k - \tilde{u}_k \right|,
\end{equation}
where \( \tilde{x}_k = x_k - x_{ss} \) and \( \tilde{u}_k = u_k - u_{ss} \) are the deviations from the steady-state.

Denote by \( K_{\mathrm{LQR}} \) and \( P_{\mathrm{LQR}} \) the optimal feedback gain and value function matrix associated with the standard LQR formulation, respectively. For this scalar system, their numerical values are:
\[
P_{\mathrm{LQR}} = 4.2361, \quad K_{\mathrm{LQR}} = 1.618.
\]

The presence of the absolute value term \( |\tilde{x}_k - \tilde{u}_k| \) in the cost function introduces nonlinearity, resulting in a piecewise structure in the value function. The corresponding critical switching boundaries in the state-input space are:
\[
\tilde{u}_k + 2\tilde{x}_k = 0, \quad \text{and} \quad \tilde{x}_k - \tilde{u}_k = 0.
\]

Using the Bellman equation:
\begin{equation}
    V(\tilde{x}_k) = \min_{\{u_k\}} \left\{ \tilde{x}_k^2 + \tilde{u}_k^2 + \left| \tilde{x}_k - \tilde{u}_k \right| + V(\tilde{x}_{k+1}) \right\},
\end{equation}
and enforcing continuity of the value function at the switching points, we obtain an approximate closed-form solution based on numerical computation:

\paragraph{Value Function}
\begin{equation}
\begin{aligned}
   & V(\tilde{x}_k) =\\
    &\begin{cases}
        5\tilde{x}_k^2 + 3|\tilde{x}_k|, & \text{if } |\tilde{x}_k| \leq 0.5, \\
        \dfrac{26\tilde{x}_k^2 + 22|\tilde{x}_k| - 1}{6}, & \text{if } 0.809>|\tilde{x}_k| >0.5\\
        P_{LQR}\tilde{x}_k^2 + 4.236|\tilde{x}_k|-0.50003, &\text{otherwise}
    \end{cases}    
\end{aligned}
\end{equation}

\paragraph{Optimal Control Policy}
\begin{equation}\label{exact}
    \tilde{u}_k =
    \begin{cases}
        -2\tilde{x}_k, & \text{if } |\tilde{x}_k| \leq 0.5, \\
        \dfrac{-10\tilde{x}_k - 1}{6}, & \text{if } 0.809>\tilde{x}_k >0.5\\
        \dfrac{-10\tilde{x}_k + 1}{6}, & \text{if } -0.809 <\tilde{x}_k <-0.5\\
        -K_{LQR}\tilde{x}_k -0.29, & \text{if } 0.809 <\tilde{x}_k \\
        -K_{LQR}\tilde{x}_k + 0.29, & \text{if } -0.809 >\tilde{x}_k 
    \end{cases}
\end{equation}

The following table presents a numerical comparison of the total cost values obtained using three different control strategies, all simulated with the initial condition \( x(0) = 12 \).
\begin{table}[]
    \centering
        \caption{Comparison of cost function types across control methods}
    \begin{tabular}{|c|c|c|c|}
    \hline
    \textbf{Method} & \textbf{Cost index \eqref{true_avg}} & \textbf{Cost index} \eqref{ex_upper}  & \textbf{LQR cost}\\
    \hline
     LQR & 420.3626  & 420.3626 & \textbf{319.5642} \\
     Controller \eqref{exact} & \textbf{420.2428}& \textbf{420.2428}& 392.9962\\
     MPC & 611.2143 & 659.0715 & 659.0715\\
     \hline
    \end{tabular}

    \label{tab:my_label}
\end{table}

\bibliographystyle{IEEEtran}
\bibliography{sample}
\end{document}